\documentclass{amsart}
\usepackage[utf8]{inputenc}
\usepackage[english]{babel}

\usepackage{amsfonts,amssymb,amsmath}
\usepackage{xcolor}
\usepackage[colorlinks,
linkcolor={red!50!black},
citecolor={blue!50!black},
urlcolor={blue!80!black}]{hyperref}
\usepackage[capitalize]{cleveref}

\usepackage[margin=30mm]{geometry}

\usepackage{tikz}
\usepackage[all]{xy}
\usepackage{enumitem}
\makeatletter
\newcommand{\mylabel}[2]{#2\def\@currentlabel{#2}\label{#1}}
\usetikzlibrary{calc, matrix, arrows, cd}

\newcommand{\bsm}{\left(\begin{smallmatrix}}
	\newcommand{\esm}{\end{smallmatrix}\right)}

\numberwithin{equation}{section}

\newtheorem*{thmintro}{Theorem}

\theoremstyle{definition}

\newtheorem*{claim*}{Claim}
\newtheorem*{remark*}{Remark}

\newcommand{\N}{\mathbb{N}}
\newcommand{\Z}{\mathbb{Z}}

\newcommand{\RP}{\mathbb{RP}}

\newcommand{\CP}{\mathbb{CP}}

\newcommand{\id}{\operatorname{id}}

\crefname{lemma}{Lemma}{Lemmas}

\title{Stably exotic knotted surfaces up to equivalence}

\author[D.~Galvin]{Daniel Galvin}
\address{Department of Mathematics, University of Texas at Austin, United States of America}
\email{daniel.galvin@austin.utexas.edu}

\author[D.~Kasprowski]{Daniel Kasprowski}
\address{School of Mathematical Sciences, University of Southampton, United Kingdom}
\email{d.kasprowski@soton.ac.uk }

\begin{document}
	\begin{abstract}
		In this short note we show the existence of stably exotic knotted surfaces in 4-manifolds in the sense that the embedded surfaces are topologically equivalent but are not smoothly equivalent even after any number of internal and external stabilisations.
	\end{abstract}
	\maketitle
	Two embeddings $s,s'$ of a surface $\Sigma$ in a $4$-manifold $M$ are topologically or smoothly equivalent if there exists a homeomorphism or diffeomorphism $g\colon M\to M$ such that $s'=g\circ s$, respectively. An \emph{external stabilisation} of $s\colon \Sigma\to M$ is given by considering the embedding $s$ as an embedding of $\Sigma$ into $M\#(S^2\times S^2)$, where the connected sum is taken away from the image of $s$. An \emph{internal stabilisation} of $s$ is the embedding $s\#U\colon \Sigma\#(S^1\times S^1)\to M\#S^4$, where $U\colon S^1\times S^1\to S^4$ is the trivial embedding. Equivalently, it is given by adding a trivial, local tube to the embedding $s$.
	
	We construct embedded surfaces that are topologically but not smoothly equivalent even after any number of external and internal stabilisations. More explicitly, we show the following.
	\begin{thmintro}
		Let $g\in\N_0$ be given.
		\begin{enumerate}
			\item There exist smoothly embedded genus $g$ orientable surfaces with normal Euler number~$1$ in $\RP^4\#\CP^2$ that are topologically but not smoothly equivalent. Moreover, they are not smoothly equivalent even after any number of internal and external stabilisations.
			\item For $n\in\N$ large enough, there exist smoothly embedded genus $g$ orientable surfaces with normal Euler number $0$ in $(\RP^2\times S^2)\#n(S^2\times S^2)$ that are topologically but not smoothly equivalent. Moreover, they are not smoothly equivalent even after any number of internal and external stabilisations.
		\end{enumerate}
	\end{thmintro}
	
	\begin{remark*}\leavevmode
		\begin{enumerate}
			\item[(i)] Since any pair of orientable, homeomorphic $4$-manifolds are stably diffeomorphic by a result of Gompf \cite{Gompf}, the ambient manifold in our examples is necessarily non-orientable. For more details on stably exotic $4$-manifolds see \cite{KP-stablyexotic} for the closed case and \cite{KOPR} for the case with boundary. In particular one can use \cite[Proposition~8.3]{KOPR} to show that our examples have exteriors that remain stably exotic after a connected sum with any number of copies of $S^1\times S^3$, thus providing an infinite family of such pairs. On the other hand for any three topologically equivalent surfaces in a given $4$-manifold at least two of them become smoothly equivalent after some number of external stabilisations as can be seen from \cite[Proposition~8.1]{KOPR}.
			\item[(ii)] The existence of stably exotic knotted surfaces in terms of isotopy, i.e.\ topologically isotopic but not smoothly stably isotopic surfaces, remains open \cite[Question~1.1]{GOP25}. For such examples to exist, the ambient manifold cannot be simply connected by \cite[Theorem 1.2]{Galvin24} and further restrictions on the fundamental group are given by \cite[Theorem~A]{GOP25}.
		\end{enumerate}
	\end{remark*}
	
	Since an internal stabilisation increases the genus of the embedded surface by one, in both cases it suffices to construct examples of embedded spheres. 
	
	\begin{proof}[Proof of (1)]
		Let $Q$ denote Cappell--Shaneson's exotic $\RP^4$. This is homeomorphic to $\RP^4$ but not stably diffeomorphic to $\RP^4$, i.e. there does not exist a $k$ such that $\RP^4\#k(S^2\times S^2)$ and $Q\#k(S^2\times S^2)$ are diffeomorphic \cite[p.61]{Cappell-Shaneson}. See \cite[Section~5.6.]{counterexamples} for a discussion of properties of $Q$ and further examples of exotic $\RP^4$s. There is a diffeomorphism $f\colon Q\#\CP^2\to \RP^4\#\CP^2$ by \cite[Theorem~2]{akbulut}, see also \cite[Theorem~18]{bais-torres}. Let $s_0\colon S^2\to \RP^4\#\CP^2$ be the embedding of $\CP^1$ into $\CP^2$. Similarly, let $s_1'\colon S^2\to Q\#\CP^2$ be the embedding of $\CP^1$ into $\CP^2$ and define $s_1:=f\circ s_1'$. By definition the exteriors of $s_0$ and $s_1$ are diffeomorphic to $\RP^4\setminus \mathring{D}^4$ and $Q\setminus \mathring{D}^4$, respectively. In particular, since the exteriors are homeomorphic, the embeddings are topologically equivalent but not externally stably smoothly equivalent. This example of exotic embedded spheres is \cite[Example~19]{bais-torres}.
		
		To complete the proof of (1) it remains to show that the embeddings $s_0$ and $s_1$ do not become externally stably diffeomorphic after any number of internal stabilisations.
		Let $s^g\colon \Sigma_g\to \CP^2$ be the internal, $g$-fold stabilisation of $\CP^1\to \CP^2$. Since the embedding is characteristic, the exterior $\CP^2_{s^g}$ is spin. Since $\CP^2_{s^g}\times I$ is a nullbordism of its boundary, $\CP^2_{s^g}\cup_\partial\CP^2_{s^g}\cong\partial(\CP^2_{s^g}\times I)$ has trivial signature. By Novikov additivity, it follows that gluing two copies of $\CP^2_{s^g}$ along any diffeomorphism of their boundaries has trivial signature. 
		
		Let $s_i^g$ be the internal, $g$-fold stabilisation of $s_i$. Their complements are diffeomorphic to $\RP^4\#\CP^2_{s^g}$ and $Q\#\CP^2_{s^g}$, respectively. In particular they admit a pin$^+$-structure. Any pin$^+$-structure will induce the unique spin-structure on $\CP^2_{s^g}$ and thus a unique spin structure on the boundary. Assume there is a stably diffeomorphism $h\colon \RP^4\#\CP^2_{s^g}\#k(S^2\times S^2)\to Q\#\CP^2_{s^g}\#k(S^2\times S^2)$. We pick pin$^+$-structures such that $h$ is pin$^+$-structure preserving. In particular, it preserves the induced spin structure on the boundary. Gluing along the boundary with $\CP^2_{s^g}$, we to obtain a diffeomorphism between the closed manifolds $\RP^4\#X_1\#k(S^2\times S^2)$ and $Q\#X_2\#k(S^2\times S^2)$, where $X_i$ is a spin manifold consisting of two copies of $\CP^2_{s^g}$ glued along the boundary. As discussed in the previous paragraph, $X_i$ has trivial signature. Since $X_i$ is also simply-connected, it is stably diffeomorphic to $l$ copies of $S^2\times S^2$ for some $l$. It follows that $RP^4$ and $Q$ are stably diffeomorphic, a contradiction. Hence $s_0$ and $s_1$ are internally and externally stably exotic as claimed.
	\end{proof}
	
	\begin{proof}[Proof of (2)]
		An application of Kreck's modified surgery yields that, for $n$ large enough, the manifolds~$X:=(S^2\times \RP^2)\#(n+11)(S^2\times S^2)$ and $X':=(S^2\times \RP^2)\#K3\#n(S^2\times S^2)$ are diffeomorphic, see also \cite[Theorem~A]{KP-stablyexotic}. Let $f\colon X'\to X$ be such a diffeomorphism. Let $s_0\colon S^2\to (S^2\times \RP^2)\#(n+11)(S^2\times S^2)$ be given by taking $S^2\times\{*\}$ in $S^2\times \RP^2$. Similarly, let $s_1'\colon S^2\to (S^2\times \RP^2)\#K3\#n(S^2\times S^2)$  be given by taking $S^2\times\{*\}$ in $S^2\times \RP^2$. We set $s_1:=f\circ s_1'$. Let $M$ be the M\"obius band. Then the complement of $\nu s_0$ is diffeomorphic to $(S^2\times M)\# (n+11)(S^2\times S^2)$ while the complement of $\nu s_1$ is diffeomorphic to $(S^2\times M)\#K3\#n(S^2\times S^2)$. Let $s_0^g$ and $s_1^g$ denote the $g$-fold internal stabilisations of $s_0$ and $s_1$. The boundary of $\nu s^g_i$ is diffeomorphic to $S^1\times \Sigma_g$ for $i=0,1$. Let $W_g$ and $W_g'$ denote the complements of $\nu s_0^g$ and $\nu s_1^g$, respectively.
		
		Assume there is a diffeomorphism taking $s_0^g$ to $s_1^g$ after taking a $k$-fold external stabilisation. Then this diffeomorphism induces a bundle isomorphism on the normal bundles and thus gives a diffeomorphism $h$ from $W_g\#k(S^2\times S^2)$ to $W_g'\#k(S^2\times S^2)$ that respects the structure of an $S^1$ bundle on the boundary. All bundle automorphism of the trivial $S^1$-bundle $S^1\times \Sigma_g$ extend over $M\times \Sigma_g$ and we obtain well-defined manifolds $X_g:=W_g\cup_\partial(M\times \Sigma_g)\#k(S^2\times S^2)$ and $X'_g:=W'_g\cup_\partial(M\times \Sigma_g)\#k(S^2\times S^2)$, where we choose the gluing map for $X'_g$ by postcomposing the gluing map for $X_g$ with $h\mid_{\partial W_g}$.  Hence we also obtain a diffeomorphism $h\cup\id\colon X_g\to X'_g$.
		 
		We now seek alternate descriptions of $X_g$ and $X'_g$ that are more amenable to the methods of modified surgery (as in \cite{KP-stablyexotic}).  Performing the internal stabilisations of $S^2\times\{*\}$ inside $S^2\times \RP^2$ yields a surface whose exterior $C$ is a manifold with fundamental group $\Z$ and boundary $S^1\times \Sigma_g$. Let $Y_g:=C\cup_\partial M\times \Sigma_g$. The inclusion of the boundary induces maps \[\pi_1(S^1\times \Sigma_g)\to \pi_1(S^1)\cong \Z\xrightarrow{\cdot 2}\Z\cong \pi_1(C)\] and \[\pi_1(S^1\times \Sigma_g)\cong \Z\oplus\pi_1(\Sigma_g)\xrightarrow{(\cdot 2,1)}\Z\oplus\pi_1(\Sigma_g)\cong \pi_1(M\times \Sigma_g).\]
		By Seifert--van Kampen it follows that $\pi_1(Y_g)$ is isomorphic to the pushout of $\Z\xleftarrow{\cdot 2}\Z\xrightarrow{\cdot 2}\Z$ which is the fundamental group of the Klein bottle. In particular, $\pi_1(Y_g)$ is 2-dimensional.
		
		By construction $X_g$ is diffeomorphic to $Y_g\#(n+k+11)(S^2\times S^2)$, while $X'_g$ is diffeomorphic to $Y_g\#K3\#(n+k)(S^2\times S^2)$. It follows again from Kreck's modified surgery theory, e.g.\ an application of \cite[Theorem~A, (iii)]{KP-stablyexotic} (using that $\pi_1(Y_g)$ is 2-dimensional), that $X_g$ and $X'_g$ are not stably diffeomorphic. This contradicts the existence of $h$. Hence $s_0^g$ and $s_1^g$ are not externally stably diffeomorphic.
	\end{proof}

	 \subsection*{Acknowledgements}
	 
	 We are grateful to Valentina Bais for suggesting to consider $\RP^4$ and $Q$ for part (1) of the theorem. We thank Patrick Orson, Mark Powell, Aru Ray and Peter Teichner for helpful discussions.
	 
	 \subsection*{AI statement}
	 
	No AI was used in the preparation of this note.
	 
	 \bibliographystyle{amsalpha}
	 \bibliography{bib}	 
\end{document}